\documentclass[11pt]{amsart}

\usepackage[T1]{fontenc}
\usepackage{lmodern,amssymb,mathtools,microtype, mathrsfs}
\usepackage{fullpage}
\usepackage[colorlinks=true,linkcolor=blue,citecolor=blue,urlcolor=blue]{hyperref}

\newcommand{\K}{\mathbb K}
\newcommand{\R}{\mathbb R}
\newcommand{\C}{\mathbb C}
\newcommand{\ud}{\,\mathrm d}
\newcommand{\nrm}[1]{\lVert#1\rVert}
\newcommand{\abs}[1]{\lvert#1\rvert}
\newcommand{\LL}{\mathscr L}
\newcommand{\DD}{\mathscr D}
\newtheorem{theorem}{Theorem}[section]
\newtheorem{corollary}[theorem]{Corollary}
\newtheorem{lemma}[theorem]{Lemma}
\newtheorem{proposition}[theorem]{Proposition}
\theoremstyle{remark}
\newtheorem{remark}[theorem]{Remark}

\newcommand{\E}{\mathbb E}
\newcommand{\PP}{\mathbb P}
\DeclareMathOperator{\sgn}{sgn}
\numberwithin{equation}{section}

\title[Hilbert transform and UMD constants]{The bounds $ \hbar_{p,X}\lesssim(\beta_{p,X})^2$ and $\beta_{p,X}\lesssim(\hbar_{p,X})^2$ are sharp}

\author{Emiel Lorist and Jan van Neerven }

\address[Emiel Lorist and Jan van Neerven]{\hfill\break\indent
Delft Institute of Applied Mathematics \hfill\break\indent
Delft University of Technology \hfill\break\indent
P.O. Box 5031 \hfill\break\indent
2600 GA Delft, The Netherlands}

\email{e.lorist@tudelft.nl}
\email{j.m.a.m.vanneerven@tudelft.nl}

\thanks{The first  author has received funding from the Veni subsidy \href{https://doi.org/10.61686/ZGRMR99948}{VI.Veni.242.057} of the Netherlands Organisation for Scientific Research (NWO)}

\usepackage[dvipsnames]{xcolor}

\begin{document}

\begin{abstract}
It was proved in the 1980s by Burkholder and Bourgain that, for any Banach space $X$ and $1<p<\infty$, the UMD$_p$ property for $X$ is equivalent to boundedness of the Hilbert transform on $L^p(\R;X)$, and that the UMD constant $\beta_{p,X}$ and the Hilbert transform constant $\hbar_{p,X}$ are related by the quadratic bounds
\begin{equation*}
 \hbar_{p,X}\lesssim(\beta_{p,X})^2,
 \qquad
 \beta_{p,X}\lesssim(\hbar_{p,X})^2.
\end{equation*}
In this paper we present examples showing that both bounds are sharp.
More precisely, we construct explicit $2^n$-dimensional
Banach spaces for which the Hilbert transform constant grows like $n$
and the UMD constant like $\sqrt n$, and a second family with the reverse
behaviour. 
\end{abstract}


\maketitle

\section{Introduction}

The equivalence between the UMD property and boundedness of the vector-valued Hilbert transform connects unconditionality of martingale differences with singular integrals, linking probability and harmonic analysis.
The classical proofs give quadratic comparisons between the two associated constants, and it has long been an open problem whether these can be improved to linear bounds. In this paper we construct explicit finite-dimensional Banach spaces showing that both quadratic dependencies are optimal.

Let $X$ be a Banach space over $\K\in\{\R,\C\}$ and let $p\in(1,\infty)$. For $f\in C_{\rm c}^1(\R;X)$, its \emph{Hilbert transform} is defined by
\begin{equation*}
 Hf(x):=\frac1\pi\lim_{\varepsilon\downarrow0}
 \int_{\abs{x-y}>\varepsilon}\frac{f(y)}{x-y}\ud y,
 \qquad x\in\R.
\end{equation*}
The \emph{Hilbert transform constant} $\hbar_{p,X}$ is the least constant $C$ such that
\begin{equation*}
 \nrm{Hf}_{L^p(\R;X)}\le C\nrm{f}_{L^p(\R;X)}
 \qquad f\in C_{\rm c}^1(\R;X).
\end{equation*}
The \emph{UMD constant} $\beta_{p,X}$ is the least constant $C$ such that
\begin{equation}\label{eq:umd-constant}
 \Big\lVert\sum_{k=1}^m\epsilon_k\,df_k\Big\rVert_{L^p(S;X)}
 \le C\Big\lVert\sum_{k=1}^m df_k\Big\rVert_{L^p(S;X)}, \qquad df_k:=f_k-f_{k-1}
\end{equation}
for every $X$-valued $L^p$-martingale $(f_k)_{k=0}^m$ adapted to a $\sigma$-finite filtration $(\mathscr F_k)_{k=0}^m$ on a $\sigma$-finite measure space $(S,\mathscr A,\mu)$ and  all scalars $\epsilon_k\in\K$ with $\abs{\epsilon_k}=1$. Finiteness of either constant for one $p\in(1,\infty)$ is equivalent to its finiteness for every $p\in(1,\infty)$. A space for which $\beta_{p,X}<\infty$ is called a {\em UMD space}.

The scalar theory of the Hilbert transform goes back to M. Riesz's theorem on conjugate functions \cite{Riesz1928}. The Banach space geometry underlying its vector-valued extension emerged from Burkholder's work on martingale transforms and his geometric characterisation of UMD spaces \cite{Burkholder1981}. Burkholder, with a contribution of McConnell, proved that the UMD property implies boundedness of the Hilbert transform \cite{Burkholder1983}. The converse was proven by Bourgain \cite{Bourgain1983}. Quantitatively, these results give the bounds
\begin{equation}\label{eq:quadratic-comparison}
 \hbar_{p,X}\leq 2(\beta_{p,X})^2,
 \qquad \beta_{p,X}\leq 2 (\hbar_{p,X})^2;
\end{equation}
see also \cite[Corollary 5.2.11 and Proposition 4.2.10]{HNVW1}. 

The question whether the powers in \eqref{eq:quadratic-comparison} can be reduced, in particular whether
\[
 \hbar_{p,X}\lesssim_p\beta_{p,X}
 \qquad\text{and/or}\qquad
 \beta_{p,X}\lesssim_p\hbar_{p,X},
\]
has been a longstanding problem; see Burkholder \cite[p. 249]{Burkholder2001} and \cite[Problem O.6]{HNVW1}. The scalar sharp constants offer a natural reason to expect a linear relation. Writing $p^*:=\max\{p,p/(p-1)\}$, the results of Pichorides \cite{Pichorides1972} and Burkholder \cite{Burkholder1984} give
\[
 \hbar_{p,\R}=\cot\frac{\pi}{2p^*},
 \qquad \beta_{p,\R}=p^*-1.
\]
These quantities, and thus also $\hbar_{p,\C}$ and $\beta_{p,\C}$,  are comparable by absolute constants for $1<p<\infty$. Indeed, they satisfy
\[
 \tfrac2\pi\beta_{p,\R} \le  \hbar_{p,\R} \le  \beta_{p,\R}.
\]
The linear dependence problem asks whether the two constants admit linear comparisons that are uniform over all Banach spaces $X$ at a fixed $p$. This differs from determining their sharp dependence on $p$ for a fixed space $X$.

We resolve both directions negatively, already at $p=2$. We construct one family of Banach spaces that has Hilbert transform constant of the order of the square of its UMD constant; a second family has the reverse behaviour. Thus neither square in \eqref{eq:quadratic-comparison} can be replaced by any smaller exponent, even if the comparison constant is allowed to depend on $p$.

\subsection{The construction and main result}

The examples are built on a finite dyadic tree, whose leaves index the coordinates of $\K^N$, and its intermediate nodes collect consecutive coordinates into blocks of successively larger size. For an integer $n\ge1$, put $N=2^n$ and let
\[
 \DD_n:=\bigcup_{s=0}^n \big\{\{(j-1)2^{n-s}+1,\ldots,j2^{n-s}\}:1\le j\le2^s\big\}.
\]
At level $s$, these blocks partition $\{1,\ldots,N\} = \{1,\ldots,2^n\}$ into $2^s$ consecutive intervals, each containing $2^{n-s}$ indices. The $\sigma$-algebras generated by these partitions form a finite dyadic filtration. 

We first introduce a norm which records the sum of the coordinates over every dyadic block. Let $H_n$ be $\K^N$ equipped with the Hilbert norm
\[
 \nrm{x}_{H_n}:= \Big(\sum_{I\in\DD_n}\Big\lvert\sum_{j\in I}x_j\Big\rvert^2\Big)^{1/2}.
\]
Thus $\nrm{x}_{H_n}$ is the Euclidean norm of the vector of dyadic block sums of $x$. Since the singleton blocks are included, this map is injective, and hence the displayed expression defines a Hilbert norm. Each coordinate belongs to exactly one block at each of the $n+1$ levels. This observation explains the factor $\sqrt{n+1}$ which repeatedly appears below.
We combine this Hilbert norm with the $\ell^1$-norm by setting $E_n:=\ell^1_N+H_n$, with norm
\begin{equation*}
 \nrm{x}_{E_n}:= \inf_{\substack{u,v\in\K^N\\u+v=x}}
 \big(\nrm{u}_{\ell^1_N}+\nrm{v}_{H_n}\big).
\end{equation*}
This decomposition lets us apply the matrix estimates below to the $\ell^1_N$ part. For the $H_n$ part, we will use that scalar $L^2$-contractions extend contractively to $H_n$-valued functions.

We will use two families of matrices in our construction. The first are the summation matrices $\Sigma_N$, which record successive partial sums of the coordinates. They are given by
\begin{equation*}
 \Sigma_N:=
 \begin{pmatrix}
 1&0&\cdots&0\\
 1&1&\ddots&\vdots\\
 \vdots&\vdots&\ddots&0\\
 1&\cdots&1&1
 \end{pmatrix}
 \end{equation*}
 The second family of matrices $D_n$ assign alternating signs to the interactions between the two halves of a dyadic block. We define these recursively by
\begin{equation}\label{eq:matricesD}
 D_0:=(1),\qquad
 D_s:=\begin{pmatrix}
 D_{s-1}&(-1)^sJ_{2^{s-1}}\\
 (-1)^sJ_{2^{s-1}}&D_{s-1} 
 \end{pmatrix},\qquad s=1,\ldots,n,
\end{equation}
where $J_d$ is the $d\times d$ matrix with all entries equal to one. In the recursion for $D_n$, each diagonal block repeats the preceding construction, while the two off-diagonal blocks couple the two halves with the common sign $(-1)^n$. 

Finally, we equip $\K^N$ with the norms
\begin{equation}\label{eq:spaces}
 \begin{split}
 \nrm{x}_{X_n}&:=\max\{\nrm{\Sigma_Nx}_{\ell^\infty_N},\nrm{x}_{E_n}\},\\
 \nrm{x}_{Y_n}&:=\max\{\nrm{D_nx}_{\ell^\infty_N},\nrm{x}_{E_n}\}.
 \end{split}
\end{equation}
The matrix term allows us to transfer the operator lower bounds to $X_n$ and $Y_n$. The upper bounds also use the Hilbert component, see Proposition \ref{prop:transfer} below.

\begin{theorem}\label{thm:main}
For every $n\ge1$, the $N$-dimensional spaces $X_n$ and $Y_n$ satisfy
\begin{align*}
 \hbar_{2,X_n}&\eqsim n,
 &\beta_{2,X_n}&\eqsim\sqrt n,\\
 \beta_{2,Y_n}&\eqsim n,
 &\hbar_{2,Y_n}&\eqsim\sqrt n.
\end{align*}
The comparison constants are universal.
\end{theorem}

Recalling that $N = 2^n$, the larger constant in each family grows like $\log N$, whereas the smaller one grows like $\sqrt{\log N}$. The proof of Theorem \ref{thm:main} has been checked in Lean 4 \cite{LN26Lean}.

\begin{remark}
Tracking constants, for every $n\ge1$ one has, over either scalar field,
\begin{align*}
 \frac n7\le\hbar_{2,X_n}&\le n+1,\\
 \sqrt{\frac n{14}}\le\beta_{2,X_n}
 &\le\min\{n+1,C_{\K}\sqrt{n+1}\},\\
 \frac{2n}3\le\beta_{2,Y_n}&\le n+1,\\
 \sqrt{\frac n3}\le\hbar_{2,Y_n}
 &\le\min\{n+1,40\sqrt{n+1}\},
\end{align*}
where one may take $C_{\R}=176$ and $C_{\C}=351$. These constants differ slightly from those in the Lean 4 formalization \cite{LN26Lean}, which uses different quantitative versions of some of the estimates.
\end{remark}

By the quantitative extrapolation theorems \cite[Theorem 4.2.7]{HNVW1} and \cite[Theorem 11.2.5]{HNVW3}, Theorem \ref{thm:main} immediately gives the following corollary.

\begin{corollary}\label{cor:all-p}
For every fixed $p\in(1,\infty)\setminus\{2\}$ and all $n\ge1$,
\begin{align*}
 \hbar_{p,X_n}&\eqsim_p n,
 &\beta_{p,X_n}&\eqsim_p\sqrt n,\\
 \beta_{p,Y_n}&\eqsim_p n,
 &\hbar_{p,Y_n}&\eqsim_p\sqrt n.
\end{align*}
The comparison constants depend only on $p$.
\end{corollary}

\subsection{Relation with earlier work and outline of the proof}

The summation operators $\Sigma_N:\ell^1_N\to\ell^\infty_N$ used in our construction were proposed by Wenzel \cite{Wenzel2004} as candidates for separating Hilbert transform and martingale transform bounds. He also studied their UMD constants. Combining our estimates with the known lower bound for the UMD constant gives
\[
 \beta_{2,\Sigma_N}\eqsim\sqrt{\log N},
 \qquad \hbar_{2,\Sigma_N}\eqsim\log N,
 \qquad N=2^n,\quad n\ge1.
\]
Section \ref{sec:reduction} defines these operator constants and transfers the estimates to $X_n$, preserving both orders of growth. Applying the same construction to the matrices $D_n$ gives the spaces $Y_n$, with the reverse separation.

Dyadic models provide another approach to the comparison problem. Petermichl \cite{Petermichl2000} represented the Hilbert transform, up to a non-zero absolute factor, as an average of translates and dilates of a Haar shift. Together with the estimates of Petermichl and Pott \cite{PetermichlPott2003}, this gives a dyadic proof of the quadratic upper bound $\hbar_{p,X}\lesssim(\beta_{p,X})^2$. More recently, Domelevo and Petermichl \cite{DP2022,DP2023,DP2026} obtained a linear comparison between the Hilbert transform and the dyadic shift defined on the Haar system of the real line by
\[
 S_0h_{I_+}=h_{I_-},\qquad S_0h_{I_-}=-h_{I_+},
\]
where $I_-$ and $I_+$ are the two children of a dyadic interval $I$. Unlike the classical shift, which sends a Haar function to a signed combination of its children, $S_0$ exchanges sibling Haar functions at the same scale. Writing $s_{p,X}$ for its norm on $L^p(\R;X)$, their results give
\[
 c_0s_{p,X}\le\hbar_{p,X}\le s_{p,X},
 \qquad c_0=\frac{8G}{\pi^2},
\]
where $G$ is Catalan's constant \cite[Theorems 2.1 and 2.2]{DP2026}. Thus the linear dependence problem reduces to comparing $S_0$ with martingale transforms. Theorem \ref{thm:main}  shows that quadratic comparisons between $s_{p,X}$ and $\beta_{p,X}$ are sharp in both directions.

For a real Banach space $X$, Geiss, Montgomery-Smith and Saksman \cite[Theorem 1.1]{GMS2010} proved the exact identities
\[
 \nrm{R_1^2-R_2^2}_{\LL(L^p(\R^2;X))} = \nrm{2R_1R_2}_{\LL(L^p(\R^2;X))} = \beta_{p,X},
\]
where $R_1,R_2$ are the Riesz transforms. 
More generally, \cite[Theorem 3.1]{GMS2010}  compares the norms of Fourier multipliers with nonconstant smooth real even symbols, homogeneous of degree zero, with the UMD constant. Furthermore, \cite[Theorem 4.1]{GMS2010}  compares the norms of Fourier multipliers with symbols $im$, where $m$ is nonzero, smooth, real, odd and homogeneous of degree zero, with the Hilbert transform constant. Combined with Theorem \ref{thm:main}, these results show that the even and odd multiplier classes can have a quadratic gap in either direction. Moreover, it was shown by Pott and Stoica \cite{PottStoica2014} that linear upper estimates  hold for sufficiently smooth Calder\'on--Zygmund operators with even kernels and cancellation. Again, the family of Banach spaces $X_n$ in Theorem \ref{thm:main} shows that such a conclusion cannot hold for general Calder\'on--Zygmund operators solely in terms of $\beta_{p,X}$.

\medskip

The proof of Theorem \ref{thm:main} uses only finite-dimensional constructions and scalar estimates. Section \ref{sec:reduction} reduces the assertions to bounds for the two matrices acting from $\ell^1_N$ to $\ell^\infty_N$. Section \ref{sec:matrix-estimates} proves these bounds by a common calculation over the dyadic tree. A martingale cancellation identity and the polarised Cotlar identity play parallel roles in the upper estimates; positivity, duality and interpolation then give the required $L^2$ bounds. Testing with explicit step functions gives the lower estimates. 

\section{Reduction to operator estimates}\label{sec:reduction}

We first separate the matrix estimates from the construction of the Banach spaces. The matrices will be viewed as operators from $\ell^1_N$ to $\ell^\infty_N$. A bound for a scalar transform followed by one of these matrices is then transferred to a bound for the same transform on $X_n$ or $Y_n$. The additional cost of this passage will be of order $\sqrt n$, which is precisely the smaller scale in Theorem \ref{thm:main}.

\subsection{Operator constants and the dyadic formulation}

Let $X,Y$ be Banach spaces, $T\in\LL(X,Y)$, and $p\in(1,\infty)$. The Hilbert transform constant with respect to $T$, denoted by $\hbar_{p,T}$, is the least constant $C$ such that
\begin{equation}\label{eq:operator-hilbert}
 \nrm{H(Tf)}_{L^p(\R;Y)}\le C\nrm{f}_{L^p(\R;X)}
 \qquad f\in C_{\rm c}^1(\R;X).
\end{equation}
The UMD constant with respect to $T$, denoted by $\beta_{p,T}$, is the least constant $C$ such that
\begin{equation}\label{eq:operator-umd}
 \Big\lVert\sum_{k=1}^m\epsilon_kT\,df_k\Big\rVert_{L^p(S;Y)}
 \le C\Big\lVert\sum_{k=1}^m df_k\Big\rVert_{L^p(S;X)}
\end{equation}
for all martingales and scalars as in \eqref{eq:umd-constant}. Taking $T$ to be the identity operator on $X$  recovers the constants of $X$.

In \eqref{eq:operator-umd}, the right-hand side is given by $\nrm{f_m-f_0}_{L^p(S;X)}$. Replacing this by $\nrm{f_m}_{L^p(S;X)}$ gives the same least constant.
For each filtration, write $E_k=\E(\,\cdot\mid\mathscr F_k)$ and $\Delta_k=E_k-E_{k-1}$. The UMD constant has the equivalent operator-norm formulation
\begin{equation}\label{eq:operator-supremum}
 \beta_{p,T}=\sup\Big\lVert T\sum_{k=1}^m\epsilon_k\Delta_k\Big\rVert_{\LL(L^p(S;X),L^p(S;Y))},
\end{equation}
where the supremum is over all measure spaces, filtrations, and scalars as in \eqref{eq:operator-umd}. Moreover, the supremum is unchanged if we restrict to $[0,1)$ with Lebesgue measure and its standard dyadic filtration. These reductions follow from the proofs of \cite[Lemma 4.2.8 and Theorem 4.2.5]{HNVW1}, which apply with $T$ in place of the identity operator.

\subsection{Estimates for dyadic block sums}

The next lemma records how the auxiliary norms $H_n$ and $E_n$ compare, and why both matrices fit the same construction. Its main geometric input is that a row of either matrix can be described using at most $n+1$ dyadic block sums. Applying the Cauchy--Schwarz inequality to these sums costs only $\sqrt{n+1}$.

\begin{lemma}\label{lem:geometry}
For $n\geq 1$, $x\in\K^N$ and $T\in\{\Sigma_N,D_n\}$ we have
\begin{align}\label{eq:En-comparison}
 \frac{1}{\sqrt{n+1}}\nrm{x}_{H_n}
 \le\nrm{x}_{E_n}&\le\min\{\nrm{x}_{\ell^1_N},\nrm{x}_{H_n}\},\\
 \nrm{Tx}_{\ell^\infty_N}&\le\sqrt{n+1}\nrm{x}_{H_n}.
\end{align}
\end{lemma}

\begin{proof}
Let $(e_j)_{j=1}^N$ be the standard basis of $\K^N$. For $x=e_j$, the block sum in the $H_n$-norm equals one on each of the $n+1$ blocks containing $j$ and zero on all other blocks. Hence $\nrm{e_j}_{H_n}=\sqrt{n+1}$. Writing $x=\sum_jx_je_j$ and applying the triangle inequality proves
\[
 \nrm{x}_{H_n}\le\sqrt{n+1}\nrm{x}_{\ell^1_N}.
\]
Therefore, for every decomposition $x=u+v$ we have
\[
 \nrm{x}_{H_n}\le\nrm{u}_{H_n}+\nrm{v}_{H_n}
 \le\sqrt{n+1}\big(\nrm{u}_{\ell^1_N}+\nrm{v}_{H_n}\big).
\]
Taking the infimum gives the lower bound in \eqref{eq:En-comparison}. The decompositions $x=x+0=0+x$ give the upper bounds.

For the second assertion, let $1\le i\le N$. The binary expansion of $i$ gives a partition of the initial segment $\{1,\ldots,i\}$ into disjoint dyadic blocks $I_1,\ldots,I_d\in\DD_n$, with $d\le n+1$.
The Cauchy--Schwarz inequality gives
\[
 \abs{(\Sigma_Nx)_i}
 =\Bigl|\sum_{s=1}^d\sum_{j\in I_s}x_j\Bigr|
 \le\sqrt{d}\,
 \Bigl(\sum_{s=1}^d\Bigl|\sum_{j\in I_s}x_j\Bigr|^2\Bigr)^{1/2}
 \le\sqrt{n+1}\,\nrm{x}_{H_n}.
\]
Next, to estimate $D_n$, for $0\le s\le n$ let $I_s(i)$ be the dyadic block of cardinality $2^{n-s}$ containing $i$. These blocks are nested, with $I_0(i)=\{1,\ldots,N\}$ and $I_n(i)=\{i\}$. Since $I_s(i)$ is one half of $I_{s-1}(i)$, the difference $I_{s-1}(i)\setminus I_s(i)$ is also a dyadic block. These differences are pairwise disjoint and partition $\{1,\ldots,N\}\setminus\{i\}$. By construction, we have $(D_n)_{ii}=1$ and $(D_n)_{ij}=(-1)^{n-s+1}$ whenever $j\in I_{s-1}(i)\setminus I_s(i)$. Therefore,
\begin{equation*}
 (D_nx)_i=x_i+\sum_{s=1}^n(-1)^{n-s+1}\sum_{j\in I_{s-1}(i)\setminus I_s(i)}x_j.
\end{equation*}
All $n+1$ blocks in this expression belong to $\DD_n$. Again by Cauchy--Schwarz
\[
 \abs{(D_nx)_i}
 \le\sqrt{n+1}\Big(\abs{x_i}^2 + \sum_{s=1}^n\Big\lvert \sum_{j\in I_{s-1}(i)\setminus I_s(i)}x_j \Big\rvert^2\Big)^{1/2} \le\sqrt{n+1}\nrm{x}_{H_n}.
\]
Taking the maximum over $i$ completes the proof.
\end{proof}

\subsection{Passing from operators to Banach spaces}
We now show how estimates for matrices acting from $\ell^1_N$ to $\ell^\infty_N$ give estimates for the Banach spaces constructed in \eqref{eq:spaces}. We will apply the next proposition with $T=\Sigma_N$ and thus $F=X_n$, or with $T=D_n$ and thus $F=Y_n$. In each case, the martingale and Hilbert transform constants of $F$ are bounded below by the corresponding operator constants of $T$. The proposition also gives upper bounds: passing from $T$ to $F$ adds at most a term of order $\sqrt n$, arising from the auxiliary norm $E_n$. 

\begin{proposition}\label{prop:transfer}
Let $T:\K^N\to\K^N$ be linear and suppose that
\[
 \nrm{T}_{\LL(\ell^1_N,\ell^\infty_N)}\le1,
 \qquad \nrm{T}_{\LL(H_n,\ell^\infty_N)}\le\sqrt{n+1}.
\]
Let $F$ be $\K^N$ equipped with the norm
\[
 \nrm{x}_F:=\max\{\nrm{Tx}_{\ell^\infty_N},\nrm{x}_{E_n}\}.
\]
Regarding $T$ as an operator from $\ell^1_N$ to $\ell^\infty_N$,
we have
\begin{equation}\label{eq:transfer-bounds}
 \begin{split}
 \beta_{2,T}&\le\beta_{2,F}
       \le\sqrt{(\beta_{2,T})^2+2(n+1)},\\
 \hbar_{2,T}&\le\hbar_{2,F}
       \le\sqrt{(\hbar_{2,T})^2+2(n+1)}.
 \end{split}
\end{equation}
\end{proposition}

\begin{proof}
By \eqref{eq:En-comparison} and the assumptions on $T$,
\begin{equation}\label{eq:F-comparison}
 \frac1{\sqrt{n+1}}\nrm{x}_{H_n}
 \le\nrm{x}_F \le\min\{\nrm{x}_{\ell^1_N},\sqrt{n+1}\nrm{x}_{H_n}\}.
\end{equation}
In particular, the identity from $\ell^1_N$ to $F$ and the map $T:F\to\ell^\infty_N$ are contractions. Their composition is the original operator $T$. Applying these two contractions before and after the defining inequalities for $\beta_{2,F}$ and $\hbar_{2,F}$ proves the lower bounds in \eqref{eq:transfer-bounds}.

Both upper bounds follow from the same estimate for contractive linear operators on scalar-valued $L^2$ spaces. Let $(S,\mathscr A,\mu)$ be a $\sigma$-finite measure space, let $U$ be a contraction on $L^2(S)$, and extend it to square integrable functions on $S$ with values in $\K^N$ by letting it act coordinatewise. Put
\[
 M:=\nrm{TU}_{\LL(L^2(S;\ell^1_N),L^2(S;\ell^\infty_N))}.
\]
Since $H_n$ is a Hilbert space, $U$ is contractive on $L^2(S;H_n)$.

Let $f$ be a simple function supported on a set of finite measure. The infimum defining $\nrm{\cdot}_{E_n}$ is attained by finite-dimensional compactness. Choosing a minimising decomposition for each value of $f$ gives simple functions $u,v$, supported on $\operatorname{supp}f$, such that $f=u+v$ and
\[
 \nrm{u(t)}_{\ell^1_N}+\nrm{v(t)}_{H_n}
 =\nrm{f(t)}_{E_n},\qquad t\in S.
 \]
For $a:=\nrm{u}_{L^2(S;\ell^1_N)}$ and $b:=\nrm{v}_{L^2(S;H_n)}$ we have
\begin{equation*}
 a^2+b^2\le
 \int_S\big(\nrm{u(t)}_{\ell^1_N}+\nrm{v(t)}_{H_n}\big)^2\ud\mu(t)
 =\nrm{f}_{L^2(S;E_n)}^2\le\nrm{f}_{L^2(S;F)}^2.
\end{equation*}
For $v$, the bound on $T:H_n\to\ell^\infty_N$ and the contractivity of $U$ on $L^2(S;H_n)$ give $\nrm{TUv}_{L^2(S;\ell^\infty_N)}\le\sqrt{n+1}\,b$. Combined with the definition of $M$, we obtain
\[
 \nrm{TUf}_{L^2(S;\ell^\infty_N)}
 \le Ma+\sqrt{n+1}\,b \le\sqrt{M^2+n+1}\nrm{f}_{L^2(S;F)}.
\]
For the other part of the $F$-norm, we use
\[
 \nrm{Uf}_{L^2(S;E_n)}
 \le\nrm{Uf}_{L^2(S;H_n)}
 \le\nrm{f}_{L^2(S;H_n)}
 \le\sqrt{n+1}\nrm{f}_{L^2(S;F)}.
\]
Combining the last two bounds gives
\begin{equation}\label{eq:sharp-transfer}
 \nrm{Uf}_{L^2(S;F)}^2 \le  \nrm{TUf}_{L^2(S;\ell^\infty_N)}^2+ \nrm{Uf}_{L^2(S;E_n)}^2 \le\big(M^2+2(n+1)\big)\nrm{f}_{L^2(S;F)}^2.
\end{equation}
By density, this estimate extends to every $f\in L^2(S;F)$.

For $U=H$ on $L^2(\R)$, the scalar isometry gives the required contraction and $M=\hbar_{2,T}$. For
$U=\sum_{k=1}^m\epsilon_k\Delta_k$, the martingale difference projections are mutually orthogonal on scalar $L^2(S)$, so $U$ is again a contraction, and \eqref{eq:operator-supremum} gives
$M\le\beta_{2,T}$. Taking the supremum over these transforms in \eqref{eq:sharp-transfer} proves the upper bounds in \eqref{eq:transfer-bounds}.
\end{proof}

\section{Estimates for the summation and dyadic matrices}
\label{sec:matrix-estimates}

We now estimate the Hilbert transform and UMD constants of $\Sigma_N,D_n:\ell^1_N\to\ell^\infty_N$. The upper bounds for $\beta_{2,\Sigma_N}$ and $\hbar_{2,D_n}$ follow from a common argument based on dyadic block sums. Scalar cancellation estimates control the terms involving different dyadic levels, leading to a cubic estimate for nonnegative inputs. An auxiliary lemma converts this estimate into the required $L^2$ operator bound. The lower bounds follow by testing the operators on explicit step functions. All implicit constants in this section are absolute.

\subsection{Scalar estimates}

Let $\Omega=[0,1)$ with Lebesgue probability measure $\PP$ and the standard dyadic filtration $(\mathscr F_k)_{k\ge0}$, i.e. $\mathscr F_k$ is generated by the intervals $[j2^{-k},(j+1)2^{-k})$, $0\le j<2^k$. Identifying a step function constant on intervals of length $2^{-n}$ with its vector of step values, $E_k$, for $0\le k\le n$, averages the coordinates over each block in $\DD_n$ of cardinality $2^{n-k}$.

For Banach spaces $X,Y$, $T\in\LL(X,Y)$, and $p\in(1,\infty)$, \eqref{eq:operator-supremum} and the mentioned dyadic reduction afterwards show that $\beta_{p,T}$ is the least constant $C$ such that
\begin{equation}\label{eq:dyadic-terminal}
 \Big\lVert\sum_{k=1}^m\epsilon_kT\Delta_kf\Big\rVert_{L^p(\Omega;Y)} \le C\nrm{f}_{L^p(\Omega;X)}
\end{equation}
for all $f\in L^p(\Omega;X)$, $m\ge1$, and scalars $\epsilon_k\in\K$ with $\abs{\epsilon_k}=1$.

The following lemma controls mixed terms that arise when a square is expanded in a cubic estimate, as we will do in the next subsection. The martingale transform is paired with an antisymmetric product, whereas the Hilbert transform is paired with a symmetric product. In the first case, products of martingale increments cancel. In the second, Cotlar's identity replaces the outer Hilbert transform by a difference of products. Both estimates are independent of the number of coordinates.

We use $L^3$ functions because their products lie in $L^{3/2}$ and can therefore be paired with another $L^3$ function in the cubic estimate. The resulting $L^3$ operator bound will yield an $L^{3/2}$ bound by duality and the desired $L^2$ bound by interpolation.

\begin{lemma}\label{lem:cancellation}
Let $m\ge1$ and $\abs{\lambda_1}=\cdots=\abs{\lambda_m}=1$, and set
\begin{equation}\label{eq:real-transform}
 U:=\sum_{k=1}^m\lambda_k\Delta_k.
\end{equation}
Let $d \geq 1$. For all $\mathscr F_m$-measurable functions $u_1,\ldots,u_d,v_1,\ldots,v_d\in L^3(\Omega)$,
\begin{equation}\label{eq:martingale-cancellation}
 \nrm{\big(U(u_jUv_j-v_jUu_j)\big)_{j=1}^d}_{L^{3/2}(\Omega;\ell^1_d)}
 \lesssim\nrm{(u_j)_{j=1}^d}_{L^3(\Omega;\ell^2_d)}
          \nrm{(v_j)_{j=1}^d}_{L^3(\Omega;\ell^2_d)}.
\end{equation}
For all $u_1,\ldots,u_d,v_1,\ldots,v_d\in L^3(\R)$,
\begin{equation}\label{eq:hilbert-cancellation}
 \nrm{\big(H(u_jHv_j+v_jHu_j)\big)_{j=1}^d}_{L^{3/2}(\R;\ell^1_d)}
 \lesssim\nrm{(u_j)_{j=1}^d}_{L^3(\R;\ell^2_d)}
          \nrm{(v_j)_{j=1}^d}_{L^3(\R;\ell^2_d)}.
\end{equation}
\end{lemma}

\begin{proof}
We start with some preparatory operator bounds on $L^3(\Omega;\ell^2_d)$.
For a  $\mathscr F_m$-measurable function $f$, write $f^*:=\max_{0\le k\le m}\abs{E_kf}$. The vector-valued Doob inequality \cite[Theorem 3.2.7]{HNVW1} gives, for every family of such functions $f_1,\ldots,f_d\in L^3(\Omega)$,
\begin{equation}\label{eq:vector-maximal}
 \nrm{(f_j^*)_{j=1}^d}_{L^3(\Omega;\ell^2_d)}
 \le\sqrt6\,\nrm{(f_j)_{j=1}^d}_{L^3(\Omega;\ell^2_d)}.
\end{equation}
By \cite[Theorem 4.2.25 and Corollary 4.5.15]{HNVW1}, we have $\nrm{U}_{\LL(L^3(\Omega))}\le \beta_{3,\K}=2$. Combining this with the Hilbert-valued extension theorem \cite[Theorem 2.1.9]{HNVW1} gives
\begin{equation}\label{eq:vector-transform}
 \nrm{(Uf_j)_{j=1}^d}_{L^3(\Omega;\ell^2_d)}
 \le2\,\nrm{(f_j)_{j=1}^d}_{L^3(\Omega;\ell^2_d)}.
\end{equation}
Furthermore, for the square function $Sf:=\big(\sum_{k=1}^m\abs{\Delta_kf}^2\big)^{1/2}$, we have by randomization, Jensen's inequality and \eqref{eq:vector-transform}
\begin{equation}\label{eq:vector-square}
 \nrm{(Sf_j)_{j=1}^d}_{L^3(\Omega;\ell^2_d)}
 \le2\,\nrm{(f_j)_{j=1}^d}_{L^3(\Omega;\ell^2_d)},
 \end{equation}
 and if $E_0f_j=0$ we have by \cite[Lemma 2.6, p.~411]{FrizZorinKranich2023}
 \begin{equation}
\label{eq:vector-square_lower}
 \nrm{(f_j)_{j=1}^d}_{L^{3/2}(\Omega;\ell^1_d)}
 \le24(1+\sqrt2)\,\nrm{(Sf_j)_{j=1}^d}_{L^{3/2}(\Omega;\ell^1_d)},
\end{equation}
where \cite[Lemma 2.4]{FrizZorinKranich2023} supplies the factor $16(1+\sqrt2)$ and Doob's inequality on $L^3$ supplies the factor $3/2$.

Fix $j\in\{1,\ldots,d\}$ and define
\begin{align*}
  u^{(k)}&:=E_ku_j, & v^{(k)}&:=E_kv_j, \\
  a^{(k)}&:=E_k(Uu_j), & b^{(k)}&:=E_k(Uv_j).
\end{align*}
Consider $w^{(k)}=u^{(k)}b^{(k)}-v^{(k)}a^{(k)}$. In the expansion of $w^{(k)}-w^{(k-1)}$, the terms containing products of two increments are $$(u^{(k)}-u^{(k-1)})(b^{(k)}-b^{(k-1)})-(v^{(k)}-v^{(k-1)})(a^{(k)}-a^{(k-1)}).$$ Since $a^{(k)}-a^{(k-1)}=\lambda_k\Delta_ku_j$ and $b^{(k)}-b^{(k-1)}=\lambda_k\Delta_kv_j$, these terms cancel, i.e. $$\Delta_ku_j\,\lambda_k\Delta_kv_j- \Delta_kv_j\,\lambda_k\Delta_ku_j=0.$$
Therefore, we obtain
\[
 w^{(k)}-w^{(k-1)} =(\lambda_ku^{(k-1)}-a^{(k-1)})\Delta_kv_j + (b^{(k-1)}-\lambda_kv^{(k-1)})\Delta_ku_j.
\]
Taking conditional expectations in the preceding identity gives $E_{k-1}w^{(k)}=w^{(k-1)}$, so $(w^{(k)})_{k=0}^m$ is a martingale. Since $E_0U=0$, we have $a^{(0)}=b^{(0)}=0$ and hence $w^{(0)}=0$. Moreover, we have $w^{(m)}=w_j:=u_jUv_j-v_jUu_j$, so $(Uw_j)_{j=1}^d$ is precisely the family we want to estimate in \eqref{eq:martingale-cancellation}.

We have $w^{(k)}=E_kw_j$ and $\Delta_k(Uw_j)=\lambda_k\Delta_kw_j$ and thus
\[
 S(Uw_j)\le S(w_j) \le\big(u_j^*+(Uu_j)^*\big)Sv_j+\big(v_j^*+(Uv_j)^*\big)Su_j.
\]
Since $E_0Uw_j=0$, \eqref{eq:vector-square_lower} gives
 \begin{align*}
 \nrm{(Uw_j)_{j=1}^d}_{L^{3/2}(\Omega;\ell^1_d)}
 \lesssim\nrm{(S(Uw_j))_{j=1}^d}_{L^{3/2}(\Omega;\ell^1_d)}.
 \end{align*}
Combined with the square-function bound,  Cauchy--Schwarz in $j$ and H\"older's inequality in $\Omega$ now give
 \begin{align*}
 \nrm{(S(Uw_j))_{j=1}^d}_{L^{3/2}(\Omega;\ell^1_d)}
 &\le\nrm{(u_j^*+(Uu_j)^*)_{j=1}^d}_{L^3(\Omega;\ell^2_d)}
       \nrm{(Sv_j)_{j=1}^d}_{L^3(\Omega;\ell^2_d)}\\
 &\quad+\nrm{(v_j^*+(Uv_j)^*)_{j=1}^d}_{L^3(\Omega;\ell^2_d)}
       \nrm{(Su_j)_{j=1}^d}_{L^3(\Omega;\ell^2_d)}.
 \end{align*}
To estimate the first term in the first product on the right-hand side, we apply \eqref{eq:vector-maximal} and \eqref{eq:vector-transform} to $(u_j)_{j=1}^d$ to obtain
\begin{align*}
 \nrm{(u_j^*+(Uu_j)^*)_{j=1}^d}_{L^3(\Omega;\ell^2_d)}
 &\le\sqrt6\,\nrm{(u_j)_{j=1}^d}_{L^3(\Omega;\ell^2_d)}
      +\sqrt6\,\nrm{(Uu_j)_{j=1}^d}_{L^3(\Omega;\ell^2_d)}\\
 &\le3\sqrt6\,\nrm{(u_j)_{j=1}^d}_{L^3(\Omega;\ell^2_d)}.
\end{align*} Combined with  \eqref{eq:vector-square}, this yields 
$$
\nrm{(u_j^*+(Uu_j)^*)_{j=1}^d}_{L^3(\Omega;\ell^2_d)}
       \nrm{(Sv_j)_{j=1}^d}_{L^3(\Omega;\ell^2_d)} \leq 6\sqrt6\,\nrm{(u_j)_j}_{L^3(\Omega;\ell^2_d)}\nrm{(v_j)_j}_{L^3(\Omega;\ell^2_d)}
$$
 Interchanging $u_j$ and $v_j$ gives the same bound for the second product.
Finally, using \eqref{eq:vector-square_lower} we can estimate
$$
\nrm{(Uw_j)_j}_{L^{3/2}(\Omega;\ell^1_d)} \leq 24(1+\sqrt2) \cdot 12\sqrt6 \,\nrm{(u_j)_j}_{L^3(\Omega;\ell^2_d)}\nrm{(v_j)_j}_{L^3(\Omega;\ell^2_d)}.
$$

For the Hilbert transform, the corresponding cancellation is the polarised Cotlar identity \cite[(5.1.23)]{Grafakos2014} 
\[
 H(uHv+vHu)=(Hu)(Hv)-uv,
\]
The scalar bound $\nrm{H}_{\LL(L^3(\R))}=\sqrt3$ \cite{Pichorides1972} and the Hilbert-valued extension theorem \cite[Theorem 2.1.9]{HNVW1} give
\[
 \nrm{(Hf_j)_{j=1}^d}_{L^3(\R;\ell^2_d)}
 \le\sqrt3\,\nrm{(f_j)_{j=1}^d}_{L^3(\R;\ell^2_d)}.
\]
Applying the identity coordinatewise, followed by Cauchy--Schwarz in $j$ and H\"older's inequality in $\R$, we obtain as before
\begin{align*}
 \nrm{\big(H(u_jHv_j+v_jHu_j)\big)_{j=1}^d}_{L^{3/2}(\R;\ell^1_d)}
 &=\nrm{\big((Hu_j)(Hv_j)-u_jv_j\big)_{j=1}^d}_{L^{3/2}(\R;\ell^1_d)}\\
 &\le\nrm{(Hu_j)_{j=1}^d}_{L^3(\R;\ell^2_d)}
          \nrm{(Hv_j)_{j=1}^d}_{L^3(\R;\ell^2_d)}
     \\&  \qquad+\nrm{(u_j)_{j=1}^d}_{L^3(\R;\ell^2_d)}
          \nrm{(v_j)_{j=1}^d}_{L^3(\R;\ell^2_d)}\\
 &\le4\,\nrm{(u_j)_{j=1}^d}_{L^3(\R;\ell^2_d)}
             \nrm{(v_j)_{j=1}^d}_{L^3(\R;\ell^2_d)}.
\end{align*}
This proves \eqref{eq:hilbert-cancellation}.
\end{proof}

\subsection{A cubic estimate}

To prove the upper bounds, it will be enough to establish a cubic estimate for nonnegative functions. The following lemma makes this precise: a weighted estimate with constant $B$ gives an operator bound of order $\sqrt B$.

\begin{lemma}\label{lem:cubic}
Let $(S,\mathscr A,\mu)$ be a $\sigma$-finite measure space, and let $V$ be a bounded linear operator on $L^3(S;\ell^2_N)$. Suppose that, for some $B\ge0$ and all $f\in L^3(S;\ell^1_N)$ with $f\ge0$, we have
\begin{equation}\label{eq:cubic-assumption}
 \int_S\sum_{i=1}^N f_i\abs{(Vf)_i}^2\ud\mu
 \le B\nrm{f}_{L^3(S;\ell^1_N)}^3.
\end{equation}
Then $ \nrm{V}_{\LL(L^3(S;\ell^1_N),L^3(S;\ell^\infty_N))}
 \lesssim\sqrt B.
$
\end{lemma}

\begin{proof}
Write $\mathscr E(f)$ for the integral in \eqref{eq:cubic-assumption}. The difficulty is that the same function $f$ appears both as a weight and inside $Vf$. To separate these roles, let $f,g\ge0$ have $L^3(S;\ell^1_N)$-norm at most one and put $h=f+g/3$. Since $Vf=Vh-(Vg)/3$ and $g_i\le3h_i$, Young's inequality $\abs{a-b/3}^2\le\frac98\abs a^2+\abs b^2$ gives
\begin{align*}
 \int_S\sum_{i=1}^N g_i\abs{(Vf)_i}^2\ud\mu
 &\le\frac{27}{8}\mathscr E(h)+\mathscr E(g)\le\frac{27}{8}\Big(\frac43\Big)^3B+B=9B.
\end{align*}
The left-hand side can now be tested against every nonnegative $g$ in the unit ball. The duality of $L^{3/2}(S;\ell^\infty_N)$ and $L^3(S;\ell^1_N)$ gives
\begin{align*}
 \nrm{Vf}_{L^3(S;\ell^\infty_N)}^2 &=\sup_{\substack{g\ge0\\\nrm{g}_{L^3(S;\ell^1_N)}\le1}}
 \int_S\sum_{i=1}^N g_i\abs{(Vf)_i}^2\ud\mu
 \le9B.
\end{align*}
Taking square roots, splitting an arbitrary complex function into the positive and negative parts of its real and imaginary parts, and rescaling proves the claim, with constant $12\sqrt B$.
\end{proof}

\subsection{Upper bounds}

We prove the upper bounds by treating the $n$ levels of the dyadic tree one at a time. Each level increases the constant in the cubic estimate by at most a fixed amount. Lemma \ref{lem:cubic} then gives an $L^3$ bound of order $\sqrt n$.

\begin{proposition}\label{prop:upper}
For every $n\ge1$, regarding $\Sigma_N$ and $D_n$ as operators from $\ell^1_N$ to $\ell^\infty_N$, we have
\begin{equation*}
 \beta_{2,\Sigma_N}\lesssim\sqrt n, \qquad\hbar_{2,D_n}\lesssim\sqrt n.
\end{equation*}
\end{proposition}

\begin{proof}
Let $K_N=(\sgn(i-j))_{i,j=1}^N$ and consider the two cases:
\begin{itemize}
 \item $R=U$, with $U$ as in \eqref{eq:real-transform}, $T=K_N$ and $\eta=-1$, with $(S,\mathscr A,\mu)=(\Omega,\mathscr F_m,\PP)$;
 \item $R=H$, $T=D_n-I$ and $\eta=1$, with $S=\R$ equipped with Lebesgue measure.
\end{itemize}
The small corrections needed to recover $\Sigma_N$ from $K_N$ and $D_n$ from $D_n-I$ will be handled at the end.

For $0\le s\le n$, put \[\DD_{n,s} :=\{I\in\DD_n:\abs I=2^{n-s}\},\]  and let $I_s(i)$ be its member containing $i$. Fix $f=(f_i)_{i=1}^N\in L^3(S;\ell^1_N)$ with $f\ge0$, and write $g=\sum_{i=1}^Nf_i$ and $f_I=\sum_{i\in I}f_i$. Positivity gives the pointwise identity $\nrm{f}_{\ell^1_N}=g$. For $T=(t_{ij})_{i,j=1}^N$, define
\[
 q_s(i):=\sum_{j\notin I_s(i)}t_{ij}f_j,
 \qquad
 \mathscr E_s:=\int_S\sum_{i=1}^Nf_i\abs{Rq_s(i)}^2\ud\mu.
\]
The function $q_s(i)$ contains the part of the $i$th matrix output coming from outside its current block. At level zero there is no such contribution, so $q_0=0$. At level $n$, the only excluded index is $i$ itself. Indeed, since the diagonal of $T$ is zero, $q_n=Tf$. Thus $\mathscr E_0=0$ and $\mathscr E_n$ is the cubic expression we want to estimate. Moreover,  $\abs{q_s(i)}\le g$ at every level, since $\abs{t_{ij}}\le1$.

Fix $1\le s\le n$ and a parent block $I\in\DD_{n,s-1}$. The function $q_{s-1}(i)$ has a common value $b_I$ for all $i\in I$. Indeed, outside $I$, every row indexed by $I$ sees the same coefficients. For $K_N$ this follows from the ordering of the indices, and for $D_n-I$ it follows from \eqref{eq:matricesD}.
In particular $\abs{b_I}\le g$. 
Let $I_-,I_+$ be the left and right children of $I$, and put
$u_I=f_{I_-}$ and $v_I=f_{I_+}$. Passing to level $s$ adds exactly the contribution of the sibling block. The entries with row index in $I_-$ and column index in $I_+$ have the constant value $c_s=-1$ in the first case and $c_s=(-1)^{n-s+1}$ in the second; the reverse entries have value $\eta c_s$, since $T^{\mathsf T}=\eta T$. Hence
\[
 q_s(i)=
 \begin{cases}
 b_I+c_sv_I,&i\in I_-,\\
 b_I+\eta c_su_I,&i\in I_+.
 \end{cases}
\]
Set $w_I=u_IRv_I+\eta v_IRu_I$. When we expand the squares, the terms involving only $Rb_I$ cancel against $\mathscr E_{s-1}$. Indeed, the new square terms and the cross terms give
\begin{equation}\label{eq:energy-increment}
\begin{split}
 \mathscr E_s-\mathscr E_{s-1}
 &=\int_S\sum_I\big(u_I\abs{Rv_I}^2+v_I\abs{Ru_I}^2\big)\ud\mu\\
 &\quad+2c_s\operatorname{Re}\int_S\sum_I(Rb_I)\overline{w_I}\ud\mu,
\end{split}
\end{equation}
where both sums run over $I\in\DD_{n,s-1}$.

The children partition all the indices. Since $u_I,v_I\le g$, the first integral is bounded by
\[
 \begin{split}
 \int_Sg\sum_{J\in\DD_{n,s}}\abs{Rf_J}^2\ud\mu
 &\le\nrm{g}_{L^3(S)}
    \nrm{(Rf_J)_{J\in\DD_{n,s}}}_{L^3(S;\ell^2(\DD_{n,s}))}^2\\
 &\lesssim\nrm{g}_{L^3(S)}
    \nrm{(f_J)_{J\in\DD_{n,s}}}_{L^3(S;\ell^2(\DD_{n,s}))}^2
 \le\nrm{g}_{L^3(S)}^3.
 \end{split}
\]
Here we used the Hilbert-valued extension theorem \cite[Theorem 2.1.9]{HNVW1} and scalar boundedness. The final inequality uses $\sum_{J\in\DD_{n,s}}f_J=g$ and nonnegativity.

For the cross term, estimating $Rb_I$ directly would retain the effect of all previous levels. Instead, use $\int_S(Ra)b\ud\mu=-\eta\int_Sa(Rb)\ud\mu$ to move $R$ onto $\overline{w_I}$. This gives
\[
 \begin{split}
 \Big\lvert\int_S\sum_I(Rb_I)\overline{w_I}\ud\mu\Big\rvert
 &=\Big\lvert\int_S\sum_Ib_IR\overline{w_I}\ud\mu\Big\rvert\\
 &\le\nrm{g}_{L^3(S)}
      \nrm{(R\overline{w_I})_I}_{L^{3/2}(S;\ell^1(\DD_{n,s-1}))}\\
 &\lesssim\nrm{g}_{L^3(S)}
      \nrm{(u_I)_I}_{L^3(S;\ell^2(\DD_{n,s-1}))}
      \nrm{(v_I)_I}_{L^3(S;\ell^2(\DD_{n,s-1}))}
 \le\nrm{g}_{L^3(S)}^3.
 \end{split}
\]
Here the proof of Lemma \ref{lem:cancellation} also bounds $U\overline{w_I}$, since $S(U\overline{w_I})\le S(w_I)$, while $w_I$ is real in the Hilbert case.
The last inequality again follows from positivity and $\sum_I(u_I+v_I)=g$. Summing \eqref{eq:energy-increment} over
the $n$ levels and using $RT=TR$ gives
\begin{equation*}
 \int_S\sum_{i=1}^Nf_i\abs{(TRf)_i}^2\ud\mu \lesssim n\nrm{f}_{L^3(S;\ell^1_N)}^3.
\end{equation*}
Since $R$ is bounded on $L^3(S;\ell^2_N)$ and $T$ is a finite matrix, $V=TR$ satisfies the boundedness assumption of Lemma \ref{lem:cubic}, which, with $B$ of order $n$, gives
\[
 \nrm{TR}_{\LL(L^3(S;\ell^1_N),L^3(S;\ell^\infty_N))} \lesssim\sqrt n.
\]
Duality gives the same bound for exponent $3/2$. Indeed, the dual operator has the same structure: $T^{\mathsf T}=\eta T$, and the scalar transform is either a martingale transform with multipliers $\overline{\lambda_k}$ or $-H$. The Riesz--Thorin theorem \cite[Theorem 2.2.1]{HNVW1}, with parameter $1/2$, now gives
\[
 \nrm{TR}_{\LL(L^2(S;\ell^1_N),L^2(S;\ell^\infty_N))} \lesssim\sqrt n.
\]

It remains to put back the elementary pieces removed at the start. We note that
\[
 \Sigma_N=\tfrac12(J_N+I+K_N),\qquad D_n=I+(D_n-I).
\]
On $L^2(\Omega)$, the orthogonality of the projections $\Delta_k$ gives $\nrm{U}_{L^2(\Omega) \to L^2(\Omega)}\le1$. The Hilbert transform is an isometry on $L^2(\R)$, see \cite[Section 5.2.a]{HNVW1}.
The contractions $\ell^1_N\hookrightarrow\ell^2_N \hookrightarrow\ell^\infty_N$, together with scalar $L^2$ contractivity, show that $R$ acts with norm at most one from $L^2(S;\ell^1_N)$ to $L^2(S;\ell^\infty_N)$. Also
\[
 \nrm{J_NUf}_{L^2(\Omega;\ell^\infty_N)}
 =\Big\lVert U\sum_{j=1}^Nf_j\Big\rVert_{L^2(\Omega)}
 \le\nrm{f}_{L^2(\Omega;\ell^1_N)}.
\]
These bounds prove the Hilbert transform estimate and the summation estimate for all unimodular multipliers. Finally, a general input on $\Omega$ may be replaced by $E_mf$, which is contractive and leaves $Uf$ unchanged. The estimate is uniform in $m$ and the multipliers, so \eqref{eq:dyadic-terminal} proves the UMD bound.
\end{proof}

\subsection{Lower bounds}
Both lower bounds use the step function that equals the \(j\)th standard basis vector $e_j$ on the \(j\)th of the \(N\) equal subintervals of \([0,1)\). For $\Sigma_N$, the calculation reduces to Hilbert transforms of interval indicators. For $D_n$, we choose the martingale signs so that the contributions from the $n$ dyadic levels have the same sign.

\begin{lemma}\label{lem:lower}
For every $n\ge1$, regarding $\Sigma_N$ and $D_n$ as operators from  $\ell^1_N$ to $\ell^\infty_N$, we have
\begin{equation*}
 \hbar_{2,\Sigma_N}\ge\frac n7, \qquad\beta_{2,D_n}\ge\frac{2n}{3}.
\end{equation*}
\end{lemma}

\begin{proof}
Both bounds use
\[
 f(t):=\sum_{j=1}^Ne_j\mathbf1_{[(j-1)/N,j/N)}(t),
\]
extended by zero outside $[0,1)$, which has $L^2(\R;\ell^1_N)$-norm one. Although this function is not smooth, it is admissible in the $L^2$ extension of \eqref{eq:operator-hilbert}.

The summation matrix joins the first $i$ coordinate intervals, so $(\Sigma_Nf)_i=\mathbf1_{[0,i/N)}$. Integration gives for $a>0$
\[
 H\mathbf1_{[0,a)}(t)=\frac1\pi\log\Big\lvert\frac{t}{t-a}\Big\rvert,
 \qquad \quad t\ne0,a.
\]
For $0<t<a$, the two terms at the truncated singularity cancel in the principal value. For $N\ge8$ and almost every $t\in[1/4,1]$, choose $i=\lfloor Nt\rfloor$. Then $1\le i\le N$ and $0<t-i/N<1/N$, whence
\[
 \nrm{\Sigma_NHf(t)}_{\ell^\infty_N} \ge\frac1\pi\log\frac{t}{t-i/N} \ge\frac1\pi\log(N/4).
\]
The maximum over coordinates is essential: it selects the endpoint $i/N$ immediately to the left of $t$. Thus the logarithmic lower bound holds throughout an interval of fixed positive measure. Integrating over $[1/4,1]$ gives $$\hbar_{2,\Sigma_N}\ge\frac{\sqrt3}{2\pi}\log(N/4)=\frac{\sqrt3\log2}{2\pi}(n-2)\ge\frac{n}7, \qquad n\ge8.$$ For $1\le n\le7$, a scalar test in the first coordinate gives $\hbar_{2,\Sigma_N}\ge1\ge n/7$.

For the UMD bound, restrict $f$ to $\Omega$ and take
\[
 U=\sum_{k=1}^n(-1)^{n-k}\Delta_k.
\]
All row averages of $D_s$ coincide. Denote their common value by $\rho_s$. A row of $D_s$ spends half its entries in a copy of $D_{s-1}$ and the other half in a constant block of sign $(-1)^s$.
Thus \eqref{eq:matricesD} gives
\[
 \rho_0=1,\qquad
 \rho_s=\frac{\rho_{s-1}+(-1)^s}{2} =\frac{(-1)^s+2^{1-s}}3,
 \qquad1\le s\le n,
\]
where the last equality follows by induction. If $f(t)=e_i$, we have
\[
 E_kf(t)=\frac1{\abs{I_k(i)}}\sum_{j\in I_k(i)}e_j.
\]
The corresponding diagonal block of $D_n$ is $D_{n-k}$, so $(D_n E_kf(t))_i=\rho_{n-k}$. The alternating signs in $U$ match the alternating principal term in these row averages. Thus,
\begin{align*}
 (D_nUf(t))_i
 &=\sum_{k=1}^n(-1)^{n-k}(\rho_{n-k}-\rho_{n-k+1})\\
 &=\frac{2n}3+\frac13\sum_{s=0}^{n-1}(-1/2)^s
 =\frac{2n}3+\frac29\big(1-(-1/2)^n\big) \ge\frac{2n}3.
 \end{align*}
There is therefore a contribution of $2/3$ at each scale, up to a bounded geometric-series correction. Since  $\nrm{f}_{L^2(\Omega;\ell^1_N)}=1$, \eqref{eq:dyadic-terminal} proves the UMD bound. Both the test functions and the transforms are real, so the same lower bounds hold over either scalar field.
\end{proof}

\begin{proof}[Proof of Theorem \ref{thm:main}]
Propositions \ref{prop:transfer} and \ref{prop:upper} give
\[
 \beta_{2,X_n}\lesssim\sqrt n, \qquad\hbar_{2,Y_n}\lesssim\sqrt n.
\]
The lower bounds in Proposition \ref{prop:transfer} and Lemma \ref{lem:lower} give
\[
 \hbar_{2,X_n}\gtrsim n, \qquad\beta_{2,Y_n}\gtrsim n.
\]
These are the four estimates for which the constructions were designed. The classical quadratic comparisons supply their reverse inequalities. For example, $\hbar_{2,X_n}\lesssim (\beta_{2,X_n})^2$ and the lower bound for $\hbar_{2,X_n}$ imply $\beta_{2,X_n}\gtrsim\sqrt n$; the upper bound for $\beta_{2,X_n}$ gives $\hbar_{2,X_n}\lesssim n$. Similarly, $\beta_{2,Y_n}\lesssim(\hbar_{2,Y_n})^2$ gives $\hbar_{2,Y_n}\gtrsim\sqrt n$ and $\beta_{2,Y_n} \lesssim n$. This proves all the assertions.
\end{proof}

\section*{AI disclosure statement}

The examples proving the quadratic dependencies were found in conversations with OpenAI's Astra model. All statements and proofs in the final manuscript have been reviewed by the authors and subsequently proof-checked in Lean 4 using OpenAI's Astra model.

\bibliographystyle{abbrv}
\bibliography{Hilbert-and-UMD-references}

\end{document}